\documentclass[a4paper,12pt]{article}
\usepackage[cp1251]{inputenc}
\usepackage[shorthands=off,english]{babel}
\usepackage{amsfonts}
\usepackage{tikz, tikz-cd}
\usepackage{amssymb,amsmath,amsthm,eucal,amscd}
\usepackage[left=1cm, right=1cm, top=1cm, bottom=2cm]{geometry}
\usepackage{mathtools}
\usepackage{enumitem}
\usepackage[figurename=Fig.]{caption}
\usetikzlibrary{patterns}
\usepackage{pgfplots}
\usepackage{relsize}
\usetikzlibrary{babel}

\usepackage{amsrefs}

\usepackage{mathtools}
\usepackage{url}
\usepackage{xcolor}
\mathtoolsset{showonlyrefs}
\newcommand{\Z}{\mathbb Z}

\newcommand{\C}{\mathbb C}
\renewcommand{\P}{\mathbb P}

\newcommand{\lb}{\lbrace}
\newcommand{\rb}{\rbrace}
\newcommand{\la}{\langle}
\newcommand{\ra}{\rangle}
\renewcommand{\phi}{\varphi}

\DeclareMathOperator{\rk}{rk}

\renewcommand{\leq}{\leqslant}
\renewcommand{\geq}{\geqslant}

\theoremstyle{plain}
\newtheorem{thm}{Theorem}[section]
\newtheorem{lm}[thm]{Lemma}
\newtheorem{cor}[thm]{Corollary}
\newtheorem{pr}[thm]{Proposition}

\theoremstyle{remark}
\newtheorem{rem}[thm]{Remark}

\theoremstyle{definition}
\newtheorem{definition}[thm]{Definition}
\newenvironment{Proof}
{\noindent{\it Proof.\/}}{{ $\Box$}\smallskip\par}

\tikzset{
  every picture/.append style={
    execute at begin picture={\shorthandoff{"}},
    execute at end picture={\shorthandon{"}}
  }
}

\begin{document}
\date{}

\author{
	Grigory Solomadin\thanks{The work was done at the Steklov Institute of Mathematics RAS and
supported by the Russian Science Foundation, grant 14-11-00414. The author is a Simons-IUM award winner and would like to thank
its sponsors and jury. Moscow State University, E-mail: \texttt{grigory.solomadin@gmail.com}}
}
\title{Quasitoric totally normally split representatives in unitary cobordism ring}

\maketitle
\abstract{The present paper generalises the results of Ray \cite{ra-86} and Buchstaber-Ray \cite{bu-ra-98}, Buchstaber-Panov-Ray \cite{bu-pa-ra-07} in unitary cobordism theory. I prove that any class $x\in \Omega^{*}_{U}$ of the unitary cobordism ring contains a quasitoric totally normally and tangentially split manifold. 
}

\section{Introduction}

In \cite{ra-86}, N. Ray gave an explicit family of stably complex manifolds representing multiplicative generators of the unitary cobordism ring $\Omega^{*}_{U}$, which are totally tangentially and normally split; i.e. whose stably tangential and normal bundles are split into the Whitney sum of complex line bundles. Both of these properties are respected under the connected sum operation. He also obtained the stably complex manifolds with the above properties representing the additively inverse elements to $\Omega^{*}_{U}$. This led to the following.

\begin{thm}\normalfont{(\cite[Theorem 3.9]{ra-86})}\label{thm:stns}
For any element of the unitary cobordism ring $\Omega_U^{*}$ of degree greater than $2$ there exists a representative which is totally tangentially and normally split.
\end{thm}

In \cite{bu-ra-98}, V.M.~Buchstaber and N.~Ray constructed toric varieties, which are polynomial generators of $\Omega^{*}_{U}$. Remind that a connected complex algebraic variety $X$ is called a \emph{toric variety} if $X$ admits an effective action of the algebraic torus $(\C^{\times})^{{\rm dim}_\C X}$ with dense open orbit (cf. \cite{bu-pa-15}). One can see that the connected sum of any two toric varieties is not toric; the corresponding obstruction is Todd genus, which takes value $1$ on any toric variety and is additive. To work around this, in \cite{bu-pa-ra-07} V.M.~Buchstaber, T.~Panov and N.~Ray introduced the box sum operation in the wider category of quasitoric manifolds (toric manifolds in sense of Davis and Januszkiewicz \cite{da-ja-91}). Given any two quasitoric manifolds, their box sum represents the class of the respective sum in $\Omega^{*}_{U}$. The ``-1''-problem of finding a manifold with additively inverse cobordism class was solved by taking the same smooth manifold with another stably complex structure (namely, the one given by the minus omniorientation matrix). Quasitoric manifolds have the property of being totally tangentially split. This led to another result.

\begin{thm}[\cite{bu-pa-ra-07}]\label{thm:q}
For any element of the unitary cobordism ring $\Omega_U^{*}$ of degree greater than $2$ there exists a representative which is a quasitoric totally tangentially split manifold.
\end{thm}

Here a natural generalisation of the above theorems is given.

\begin{thm}\label{thm:main}
For any element of the unitary cobordism ring $\Omega_U^{*}$ of degree greater than $2$ there exists a representative which is a quasitoric totally tangentially and normally split manifold.
\end{thm}

The manifolds discussed above are shown in Section \ref{sec:gens} to be not totally normally split, in general. So, to prove Theorem \ref{thm:main}, some methods are developed.

First, a globalisation of bounded flag manifold is given in Section \ref{sec:bffb}; the necessary data are a complex manifold and an ordered tuple of complex line bundles over the latter. A formula for the Milnor number of the corresponding stably complex manifold is given.

Second, a method of producing new totally normally split toric varieties is given in Section \ref{sec:gens}. Namely, this is blow-up of a totally normally split toric variety at an invariant complex codimension $2$ subvariety.

These are used to construct some totally normally split toric varieties which are then shown to be multiplicative generators of $\Omega_U^{*}$. Finally, a possible adaptation of N.~Ray's construction to Theorem \ref{thm:main} is discussed in Section \ref{sec:conc}.


The author is grateful to V.M.\,Buchstaber and T.E.\,Panov for numerous fruitful discussions.

\section{Bounded flag fibre bundles}\label{sec:bffb}


The idea of the bounded flag manifold (\cite{bu-ra-98'}, \cite[\S 7.7]{bu-pa-15}) can be globalised in terms of fiber bundles. In this Section, the corresponding construction is given. For the rest of this Section, $X$ stands for a compact stably complex smooth manifold of real dimension $2n$ and $\xi_{i}\to X,\rk \xi_{i}=1, i=1,\dots,k+1$, are complex linear vector bundles over $X$. Also let $\xi:=\bigoplus_{i=1}^{k+1}\xi_{i}$. Everywhere below pull-backs and tensor products of vector bundles are omitted, unless otherwise specified.

\subsection{Definition and properties}
We start this Section with an inductive definition
\begin{definition}\label{def:bott_glob}
Let $BF(\xi_{1}):=X,\zeta_{1}:=\xi_{1}\to BF(\xi_{1})$. For $k\geq 1$ $BF(\xi_{k+1},\dots,\xi_{1})$ is by definition a $\C P^{1}$-bundle over $BF(\xi_{k},\dots,\xi_{1})$. Namely, one has:
\begin{equation}\label{eq:bffbdef}
BF(\xi_{k+1},\dots,\xi_{1})=\P (\zeta_{k}\oplus\xi_{k+1})\to BF(\xi_{k},\dots,\xi_{1}).
\end{equation}
$\zeta_{k+1}$ is by definition the tautological line bundle over $BF(\xi_{k+1},\dots,\xi_{1})$.
\end{definition}

\begin{rem}
Topology and complex structure of the bounded flag bundle $BF(\xi_{k+1},\dots,\xi_{1})$, generally speaking, depend on the order of the tuple $(\xi_{k+1},\dots,\xi_{1})$.
\end{rem}


The bounded flag bundle $BF(\xi_{n+1},\dots,\xi_{1})$ has the natural complex (algebraic, toric, resp.) structure coming from the well-known complex (algebraic, toric, resp.) structure on projectivisations of complex vector bundles over complex (algebraic, toric, resp.) manifolds (non-singular varieties, resp.). (See, for example, \cite{so-us-16}.) The natural complex structure on the bounded flag bundle $BF(\xi_{n+1},\dots,\xi_{1})$ is given by the formula:

\begin{equation}\label{eq:cstr}
TBF(\xi_{k+1},\dots,\xi_{1})\oplus\underline{\C}^{k}\simeq \bigoplus_{i=1}^{k}(\zeta_{i}\oplus\xi_{i+1})\overline{\zeta_{i+1}}\oplus TX,
\end{equation}
where $\underline{\C}$ is the linear trivial vector bundle over $BF(\xi_{k+1},\dots,\xi_{1})$ and $\overline{\zeta_{i+1}}$ is complex conjugate to $\zeta_{i+1}$.

Let $\zeta_{k+1}^{*}\to BF(\xi_{k+1},\dots,\xi_{1})=\P(\zeta_{k}\oplus\xi_{k+1})\to BF(\xi_{k},\dots,\xi_{1})$ be the vector bundle s.t. the corresponding fiber at a point $l\subset(\zeta_{k}\oplus\xi_{k+1})_{b}, b\in BF(\xi_{k},\dots,\xi_{1})$, consists of the vectors from the line $l^{\perp}$ orthogonal to $l$. ($\zeta_{k}\oplus\xi_{k+1}$ is endowed with a hermitian metric as a vector bundle over a compact manifold.) Define $\zeta^{*}_{i}\to BF(\xi_{k+1},\dots,\xi_{1})$ as the pull-back of the corresponding bundle $\zeta^{*}_{i}\to BF(\xi_{i},\dots,\xi_{1})$ under the composition $BF(\xi_{k+1},\dots,\xi_{1})\to BF(\xi_{i},\dots,\xi_{1})$ of projection maps \eqref{eq:bffbdef}, $i=1,\dots,k+1$. The vector bundle $\zeta_{i}^{*}\to BF(\xi_{k+1},\dots,\xi_{1})$ is linear and satisfies the identity (cf. \cite{ra-86}):
\begin{equation}\label{eq:1.13}
\zeta_{i+1}\oplus\zeta_{i+1}^{*}\simeq\zeta_{i}\oplus\xi_{i+1},\ i=1,\dots,k.
\end{equation}
Hence,
\begin{equation}\label{eq:14}
\zeta_{i+1}\oplus\bigoplus_{k=1}^{i}\zeta_{k+1}^{*}\simeq\bigoplus_{k=1}^{i+1}\xi_{i},
\end{equation}\label{eq:3}
where $i=1,\dots,k$.

Remind that for a complex vector bundle $\alpha\to X$ there exists a stably inverse complex vector bundle $\theta\to X$, i.e. $\alpha\oplus\theta\simeq\underline{\C}^{r}$, where $\underline{\C}^{r}$ is the trivial complex vector bundle of rank $r$ over $X$. It is unique up to stable isomorphism. The following Lemma uses the argument similar to the one from \cite{ra-86}.

\begin{lm}\label{lm:66}
Let $\alpha,\alpha'\to X$ be complex linear vector bundles whose stably inverse complex vector bundles are totally split. Let $f:Y\to X$ be a continious map. Then the stably inverse complex vector bundles to $\overline{\alpha},f^{*}\alpha,\alpha\oplus\alpha',\alpha\alpha'$ are totally split.
\end{lm}
\begin{proof}
The claim of the Lemma is straightforward to check for $\overline{\alpha},f^{*}\alpha,\alpha\oplus\alpha'$. Let $\theta=\bigoplus_{i=1}^{m}\theta_{i},\theta'=\bigoplus_{i=1}^{m'}\theta'_{i}\to X,\rk_{\C}\theta_{i}=\rk_{\C}\theta'_{i}=1$, be complex totally split vector bundles such that
\[
\alpha\oplus\theta=\underline{\C}^{m+1},\alpha'\oplus\theta'=\underline{\C}^{m'+1}.
\]
Taking tensor product of the above two vector bundles one obtains:
\[
\alpha\alpha'\oplus\bigl(\alpha\theta'\oplus\alpha'\theta\oplus\theta\theta'\bigr)=\underline{\C}^{(m+1)(m'+1)}.
\]
It remains to observe that $\alpha\theta'\oplus\alpha'\theta\oplus\theta\theta'$ is a Whitney sum of complex linear vector bundles.
\end{proof}

\begin{pr}\label{pr:split}
Suppose that $X$ is totally normally split and stably inverses to $\xi_{i}\to X$ are totally split for all $i=1,\dots,k+1$. Then $BF(\xi_{k+1},\dots,\xi_{1})\to X$ is totally normally split.
\end{pr}
\begin{proof}
The identity \eqref{eq:14} and the property of being totally split for $\xi_{i}$'s implies that $\zeta_{i}\to BF(\xi_{k+1},\dots,\xi_{1})$ is stably normally splitting for $i=1,\dots,k+1$. Now the claim follows from the formula \eqref{eq:cstr} and Lemma \ref{lm:66}.
\end{proof}
Remind that $s_{n}(X)$ is the Milnor number of the stably complex manifold $X$:
\[
s_{n}(X^{n})=\la t_{1}+\cdots+t_{k},[X^{2n}]\ra\in\Z,
\]
where $t_{1},\dots,t_{k}$ are Chern roots of $X$ and the coupling is the natural one $H^{2n}(X;\Z)\times H_{2n}(X;\Z)\to\Z$. Let $c_{1}(\xi_{i})=x_{i}\in H^{2}(X;\Z)$.
\begin{pr}\label{pr:snbfb}
If $n+k$ is even, then $s_{n+k} (BF(\xi_{k+1},\dots,\xi_{1}))=0.$ Otherwise,
\[
s_{n+k}(BF(\xi_{k+1},\dots,\xi_{1}))=2\la(1+x_{k+1})^{n+k-1}(1+x_{k})^{-1}\cdots (1+x_{1})^{-1},[X^{n}]\ra.
\]
\end{pr}
\begin{proof}
Let $X_{n+i}:=BF(\xi_{i+1},\dots,\xi_{1}), i=0\dots,k$. One has (see \cite{so-us-16})
\[
T X_{n+k}\oplus\underline{\C}\simeq (\zeta_{k}\oplus\xi_{k+1})\overline{\zeta_{k+1}}\oplus TX_{n+k-1}.
\]
Let $y_{i}:=c_{1}(\overline{\zeta_{i}})$. Then
\begin{equation}\label{eq:technical}
s_{n+k}(X_{n+k})=\la (y_{k+1}+x_{k+1})^{n+k}+(y_{k+1}-y_{k})^{n+k},X_{n+k}\ra.
\end{equation}
I use Segre class (i.e. the multiplicatively inverse to Chern class, see \cite{so-us-16}) to compute the summands in the above expression:
\begin{multline}\label{eq:2.1}
\la (y_{k+1}+x_{k+1})^{n+k},X_{n+k}\ra=\la (1+x_{k+1})^{n+k-1}(1-y_{k})^{-1},X_{n+k-1}\ra=\\
=\la (1+x_{k+1})^{n+k-1}(1+x_{k})^{-1}(1-y_{k-1})^{-1},X_{n+k-2}\ra=\dots=\\
=\la (1+x_{k+1})^{n+k-1}(1+x_{k})^{-1}\dots(1+x_{3})^{-1}(1-y_{2})^{-1},X_{n+1}\ra=\\
=\la (1+x_{k+1})^{n+k-1}(1+x_{k})^{-1}\dots(1+x_{2})^{-1}(1+x_{1})^{-1},X^{n}\ra.
\end{multline}
\begin{multline}\label{eq:2.2}
\la (y_{k+1}-y_{k})^{n+k},X_{n+k}\ra=\la (1-y_{k})^{n+k-1}(1+x_{k+1})^{-1},X_{n+k-1}\ra=\\
=\la\sum_{i=0}^{n+k-1}\binom{n+k-1}{i}(-y_{k}^{i})(-x_{k+1})^{n+k-1-i},X_{n+k-1}\ra=\\
=(-1)^{n+k-1}\la (1+x_{k+1})^{n+k-1}(1+x_{k})^{-1}(1-y_{k})^{-1},X_{n+k-1}\ra=\dots=\\
=(-1)^{n+k-1}\la (1+x_{k+1})^{n+k-1}(1+x_{k})^{-1}\dots(1+x_{2})^{-1}(1+x_{1})^{-1},X^{n}\ra.
\end{multline}
Substituting up the identities \eqref{eq:2.1},\eqref{eq:2.2} into \eqref{eq:technical} one obtains the claim of the Proposition.
\end{proof}

\subsection{Bounded flag varieties}

Let $X=pt$, so all $\xi_{i}$ are trivial linear vector bundles, $i=1,\dots,n+1$. Then $BF(\xi_{n+1},\dots,\xi_{1})= BF(\underbrace{\underline{\C},\dots,\underline{\C}}_{n+1})$ is a bounded flag toric variety of complex dimension $n$. The Proposition \ref{pr:split} implies that $BF_{n}$ is totally normally split. There is a natural projection $BF_{n}\to BF_{n-1}$. Denote the vector bundles $\zeta_{i+1},\zeta_{i+1}^{*}\to BF_{n}$ by $\beta_{i},\beta_{i}^{*},i=0,\dots,n$ resp. The identity \eqref{eq:1.13} becomes
\begin{equation}\label{eq:1.5}
\beta_{i+1}\oplus\beta_{i+1}^{*}\simeq\beta_{i}\oplus\underline{\C},
\end{equation}
and the identity \eqref{eq:14} becomes
\begin{equation}\label{eq:2}
\beta_{i}\oplus\bigoplus_{k=1}^{i}\beta_{k}^{*}\simeq\underline{\C}^{i+1},
\end{equation}
where $i=0,\dots,n$ (see \cite{ra-86}).
\begin{cor}\label{cor:bfsplit}
$BF_{n}$ is totally normally split.
\end{cor}
\begin{proof}
Follows from Proposition \ref{pr:split}.
\end{proof}

Let $e_{0},\dots,e_{n}$ and $(z_{0},\dots,z_{n})$ be a basis and the corresponding coordinates in $\C^{n+1}$, resp. Let $\C_{i}:=\C\la e_{i}\ra,i=0,\dots,n$, where $\C\la e_{i}\ra$ is the linear hull of the vector $e_{i}$. Then $BF_{n}$ is identified with the set of tuples of lines $( l_{0},\dots,l_{n}), l_{0}:=\C_{0}$, lying in $\C^{n+1}$ and satisfying the identities
\begin{equation}\label{eq:1}
l_{i+1}\subset l_{i}\oplus \C_{i+1},\ i=0,\dots,n-1.
\end{equation}
It follows that there is inclusion
\begin{equation}\label{eq:incl}
l_{i}\subset\C\la e_{0},\dots, e_{i}\ra
\end{equation}
for $i=0,\dots,n$. Thus, $BF_{n}$ can be considered as the subvariety of $\prod_{i=1}^{n}\C P^{i}$ given by the (algebraic) conditions \eqref{eq:1}. Let $[z_{i,0}:\dots:z_{i,i}]$ be the homogeneous coordinates of the $i$-th factor of $\prod_{i=1}^{n}\C P^{i}$. Then the standard linear algebra theorem implies that the conditions \eqref{eq:1} are exactly the vanishing equations for all $2\times 2$-minors of the matrices:
\begin{equation}\label{eq:bfeq}
\rk
\left(
\begin{matrix}
z_{i+1,0} & z_{i+1,1} & \dots & z_{i+1,i}\\
z_{i,0} & z_{i,1} & \dots & z_{i,i}
\end{matrix}
\right)=1,
\end{equation}
where $i=1,\dots,n-1$.

Let $f_{n}$ be the map
\begin{equation}\label{eq:fi}
f_{n}:\ BF_{n}\to\C P^{n},\ (l_{0},\dots, l_{n})\mapsto l_{n},
\end{equation}
i.e. the restriction of the projection morphism $\prod_{i=1}^{n}\C P^{i}\to \C P^{n}$ to $BF_{n}$. Then
\begin{equation}\label{eq:2.4}
f_{n}^{*}\eta_{n}=\beta_{n},
\end{equation}
where $\eta_{n}$ is the tautological line bundle over $\C P^{n}$. One can show that $f_{n}$ is a composition of sequential blow-ups at the strict transforms of the projective subspaces $\lb z_{0}=\dots=z_{n-1}=0\rb\subset\dots\subset\lb z_{0}=z_{1}=0\rb\subset\C P^{n}$ given in the homogeneous coordinates $[z_{0}:\dots:z_{n}]$ of $\C P^{n}$.
\begin{rem}
The above description of the bounded flag variety differs from the standard one (cf. \cite{bu-ra-98}, \cite[p. 292]{bu-pa-15}) which is denoted by $BF'_{n}$. However, these varieties are isomorphic. Firstly, one has to change the basis of $\C^{n+1}$ to $\lb v_{1},\dots,v_{n+1}\rb,v_{i} = e_{n+1-i}, i=0,\dots,n$. Then \eqref{eq:1} becomes
\[
l_{i+1}\subset l_{i}\oplus \C\la v_{n-i}\ra
\]
for $i=0,\dots,n$. Secondly, the change of the order of the lines $L_{i}=l_{n-i+1}$ leads to
\[
L_{n-i}\subset L_{n-i+1}\oplus \C\la v_{n-i}\ra
\]
for $i=0,\dots,n$. It remains to substitute the parameter $i=j-1$ in the above to obtain the definition of $BF_{n}'$.
\end{rem}



\section{Constructing generators of $\Omega^{*}_{U}$}\label{sec:gens}

In this Section I explain why the known generators of the unitary cobordism ring $\Omega^{*}_{U}$ are inappropriate for the proof of Theorem \ref{thm:main}. I introduce some non-singular projective toric varieties $X_{i,j}, 0\leq i\leq j$, of complex dimension $i+j$. These are bounded flag fibre bundles over $BF_{i}$ (hence, Bott towers). $X_{i,j}$ are totally normally split. However, $X_{i,j}$ lack the necessary properties of their respective Milnor numbers. A construction of the modification in the class of stably normally splitting toric varieties is then given. It is applied to $X_{i,j}$ to obtain the manifolds $M_{i,j}$ which are (together with $BF_{i+j}$) shown in the following Sections to be the quasigenerators of $\Omega^{U}_{*}$ having all the necessary properties. The pullback of the vector bundle $\xi\to Y$ under the natural projection map $X\times Y\to Y$ is denoted below with a prime, $\xi'\to X\times Y$.

\subsection{Milnor hypersurfaces}

Remind that Milnor hypersurface $H_{i,j}, 0\leq i\leq j,$ is the $(1,1)$-bidegree hypersurface of $\C P^{i}\times \C P^{j}$ (of complex dimension $i+j-1$) given by the equation
\begin{equation}\label{eq:mi}
\sum_{k=0}^{i} z_{k}w_{k}=0
\end{equation}
in the homogeneous coordinates $[z_{0}:\dots: z_{i}], [w_{0}:\dots: w_{j}]$ of $\C P^{i}$ and $\C P^{j}$, respectively. (Notice, that $H_{0,0}=\varnothing$.) 
The fiber of the tautological line bundle $\eta_{i}\to\C P^{i}$ over $[l]\in\C P^{i}, l\subset V$, consists by definition of all vectors $v\in l$. Let $\eta_{i}^{*}$ be the vector bundle over $\C P^{i}$ with the fiber over $[l]\in\C P^{i}, l\subset V$, consisting of the vectors from the orthogonal complement $l^\perp$ in $\C^{i+1}$. One has $\rk \eta_{i}^{*}=i$ and $\eta_{i}\oplus\eta_{i}^{*}=\underline{\C}^{i+1}$.

\begin{pr}\label{pr:mh}
The restriction of the projection map $\C P^{i}\times \C P^{j}\to \C P^{i}$ induces the fibre bundle structure
\[
H_{i,j}=\P (\overline{\eta_{i}^{*}}\oplus \underline{\C}^{j-i})\to \C P^{i}.
\]
\end{pr}
\begin{Proof}
Consider a point $(z,w)=([z_{0}:\dots: z_{i}], [w_{0}:\dots: w_{j}])\in H_{i,j}$. The fiber of the map $\pi_{i,j}$ at a point $\pi(z,w)=z$ consists of the points $[w_{0}:\dots: w_{j}]$ with arbitrary values of $w_{k}, k=0,\dots, j-i-1$. The values of $w_{k}, k=j-i,\dots, j$, are such that the vector $(w_{j-i},\dots,w_{j})\in(\C^{i+1})^{\times}$ is non-zero and belongs to the conjugate vector space to the orthogonal complement of the line $[z]\subset \C^{i+1}$. Hence, the fiber of the projection map at over $z$ is identified with the corresponding fiber of the desired projective bundle.
\end{Proof}
The Milnor numbers of $H_{i,j}$ and complex projective spaces are easily computed. This leads to
\begin{pr}[\cite{mi-65}]\label{pr:qgmi}
$H_{i,j},0\leq i\leq j,i+j=n+1$, constitute a family of multiplicative generators of the ring $\Omega^{*}_{U}$ in degree $2n$.
\end{pr}
Notice that $H_{0,j}=\C P^{j-1}$ and $H_{1,j}=\P (\eta_{1}\oplus \underline{\C}^{j-1})\to \C P^{1}$ are toric varieties. In fact, this is the only case for Milnor hypersurfaces. (The cohomology ring $H^{*}(H_{i,j};\Z)$ of $H_{i,j}$ is the obstruction to be a toric variety.)
\begin{pr}\normalfont{(\cite[Theorem 9.1.5]{bu-pa-15})}
Let $1\leq i\leq j;\ 2<j$. Then $H_{i,j}$ is not a toric variety.
\end{pr}
A well-known construction in unitary cobordism theory is the dualization $D$ of the complex linear vector bundle $\xi\to X$ over a complex compact manifold (see \cite[p.78]{st-68},\cite{th-54}). $D$ is a certain stably complex submanifold of $X$ of (real) codimension $2$. The normal complex linear vector bundle of the inclusion $D\subset X$ coincides with the restriction of $\xi$ to $D$. In topological terms, the Poincar\'e duality sends the first Chern class $c_{1}(\xi)\in H^{2}(X;\Z)$ to a certain (real) codimension $2$ subvariety $D$ of $X$ representing the dual cycle $[D]\in H_{2(n-1)}(X;Z)$. (Also this construction can be formulated in the category of algebraic varieties in terms of the cycle map $Pic(X)\to Cl(X)$.)
\begin{pr}\normalfont{(\cite{mi-65},\cite[Proposition D.6.3]{bu-pa-15})}
$H_{i,j}$ is the dualization of the complex linear vector bundle $\overline{\eta_{i}}\overline{\eta'_{j}}$.
\end{pr}

\subsection{Buchstaber-Ray varieties}
In \cite{bu-ra-98}, V.M.~Buchstaber and N.~Ray introduced smooth projective toric varieties $BR_{i,j}$ to construct multiplicative generators of the unitary cobordism ring $\Omega^{*}_{U}$. (The original denotation for these manifolds from \cite{bu-ra-98} is $B_{i,j}$. I replace it to avoid confusion with the denotation of \cite{ra-86}.)

\begin{definition}[See \cite{bu-ra-98}]\label{def:brij}
Let $0\leq i\leq j$. Then $BR_{i,j}=\P (f_{i}^{*}\overline{\eta_{i}^{*}}\oplus\underline{\C}^{j-i})\to BF_{i}$ is the pullback of $H_{i,j}=\P (\overline{\eta_{i}^{*}}\oplus \underline{\C}^{j-i})\to \C P^{i}$ under the map $f_{i}:BF_{i}\to\C P^{i}$. In particular, $BR_{0,j}=\C P^{j-1}$. (Notice, that $BR_{0,0}=\varnothing$.)
\end{definition}

\begin{pr}
$BR_{i,j}\subset BF_{i}\times \C P^{j}$ is a dualization of the linear vector bundle $\overline{\beta_{i}}\overline{\eta'_{j}}$.
\end{pr}
\begin{Proof}
Remind that the normal bundle of $H_{i,j}\subset\C P^{i}\times\C P^{j}$ is equal to the corresponding restriction of $\overline{\eta_{i}}\overline{\eta'_{j}}$. Hence, by the identity \eqref{eq:2.4} the normal bundle of $BR_{i,j}\subset BF_{i}\times\C P^{j}$ equals to the restriction of $(f_{i}^{*}\overline{\eta_{i}})\overline{\eta'_{j}}=\overline{\beta_{i}}\overline{\eta'_{j}}$. Q.E.D.
\end{Proof}
$BR_{i,j}$ is the preimage of the Milnor hypersurface $H_{i,j}\subset\C P^{i}\times\C P^{j}$ under the map $f_{i}\times Id_{j}:BF_{i}\times\C P^{j}\to \C P^{i}\times\C P^{j}$. Hence, $BR_{i,j}$ is the hypersurface in $BF_{i}\times \C P^{j}$ given by the equation:
\begin{equation}\label{eq:br}
\sum_{k=0}^{i} z_{i,k}w_{k}=0,
\end{equation}
where $[w_{0}:\dots:w_{j}]\in\C P^{j}$ are the homogeneous coordinates on $\C P^{j}$ and $z_{k,l}$ are the coordinates on $BF_{i}$. There is a map (birational morphism) $BR_{i,j}\to H_{i,j}$ which is a restriction of $(f_{i}, Id_{j}):\ BF_{i}\times \C P^{j}\to \C P^{i}\times \C P^{j}$ and has degree $1$. Using the properties of the Milnor number $s_{i+j-1}$ (\cite[Theorem 9.1.8]{bu-pa-15}) and Proposition \ref{pr:qgmi} one justifies

\begin{pr}
\[
s_{i+j-1}(BR_{i,j})=s_{i+j-1}(H_{i,j}).
\]
The varieties $BR_{i,j},0\leq i\leq j;\ i+j=n+1$, are multiplicative generators of the ring $\Omega^{*}_{U}$ in degree $2n$.
\end{pr}
\begin{Proof}
\end{Proof}

The varieties $BR_{i,j}$ as well as $H_{i,j}$ are spaces of projective fiber bundles. Let's specify this structure.

\begin{lm}\label{lm:pb2}
For any $k=1,\dots,i$ one has an identity over $BF_{i}$:
\[
f_{i}^{*}\eta_{k}^{*}\simeq\bigoplus_{q=1}^{k}\beta_{q}^{*}.
\]
\end{lm}
\begin{Proof}
By the identity \eqref{eq:2.4} one has $f_{i}^{*}\eta_{k}\simeq\beta_{k}$. Next, $f_{i}^{*}(\underline{\C}^{k+1})=\underline{\C}^{k+1}$. Consequently, $f_{i}^{*}\eta_{k}^{*}\oplus\beta_{k}\simeq\underline{\C}^{k+1}$. It remains to use formula \eqref{eq:2}.
\end{Proof}

\begin{pr}\label{pr:br}
\[
BR_{i,j}=\P(\bigoplus_{k=1}^{i}\overline{\beta_{k}^{*}}\oplus\underline{\C}^{j-i})\to BF_{i}.
\]
\end{pr}
\begin{Proof}
Follows from Lemma \ref{lm:pb2}, Proposition \ref{pr:mh} and Definition \ref{def:brij}.
\end{Proof}

Unlike Milnor hypersurfaces, the Buchstaber-Ray varieties are toric.

\begin{cor}[See \cite{bu-ra-98}]\label{cor:toric}
$BR_{i,j}$ is a non-singular projective toric variety.
\end{cor}

\begin{rem}\label{rem:breq}
Let $0\leq i\leq j$. The equation \eqref{eq:br} is not invariant under the usual torus action on $BF_{i}\times\C P^{j}$, unless $i=0$. Hence, for $1\leq i\leq j$, $BR_{i,j}$ is not an invariant divisor of $BF_{i}\times \C P^{j}$. However, due to formula \eqref{eq:2} one can identify the trivial $\C P^{j}$-bundle over $BF_{i}$ with the projectivisation:
\[
\P(\overline{\beta_{i}}\oplus\bigoplus_{k=1}^{i}\overline{\beta_{k}^{*}}\oplus\underline{\C}^{j-i})\to BF_{i}.
\]
Now endow $BF_{i}\times\C P^{j}$ with the $(\C^{\times})^{n}$-action coming from the above fiber bundle structure. (It is equivariantly isomorphic to the previous one. This follows, for example, from the main result of \cite{flo-03}.) The embedding
\[
BR_{i,j}=\P(\bigoplus_{k=1}^{i}\overline{\beta_{k}^{*}}\oplus\underline{\C}^{j-i})\subset
 \P(\overline{\beta_{i}}\oplus\bigoplus_{k=1}^{i}\overline{\beta_{k}^{*}}\oplus\underline{\C}^{j-i})=BF_{i}\times\C P^{j}
\]
is equivariant.
\end{rem}

It is important that $BR_{i,j}$ are not totally normally split, in general. To prove this fact, an auxiliary statement is necessary.

\begin{lm}\normalfont{(\cite[Theorem 1.5]{mit-06}).}\label{lm:cpsp}
Let $n>1$. Then $\C P^{n}$ is not totally normally split.
\end{lm}

Notice that $\C P^{1}$, as well as any Riemannian surface $\Sigma_{g}$ of arbitrary genus $g$, is totally normally split. This follows from the fact that any complex vector bundle over $\Sigma_{g}$ is topologically isomorphic to the Whitney sum of a complex linear vector bundle and a trivial one.

\begin{pr}\label{pr:nsp}
Let $\xi\to X$ be a complex vector bundle of complex rank $k>2$ over a smooth compact stably complex manifold $X$. Then the fiberwise projectivisation $\P(\xi)\to X$ is not totally normally split.
\end{pr}
\begin{Proof}
Assume the contrary, i.e. the sum $N\P(\xi)\oplus\underline{\C}^{q}$ is isomorphic to a sum of complex linear vector bundles for some $q\geq 0$. Let $x\in X$. Denote the corresponding fiber inclusion map by $\iota:\C P^{k-1}_{x}\to \P (\xi)$. Then the pull-back $\iota^{*}(N\P(\xi)\oplus\underline{\C}^{q})$ splits into a sum of line bundles. By definition, $\iota^{*}N\P(\xi)=N\C P^{k-1}_{x}$. Hence, we obtain a contradiction to Lemma \ref{lm:cpsp}. Q.E.D. 
\end{Proof}

In \cite{lu-pa-14}, Z.~L\"{u} and T.~Panov introduced another family of multiplicative generators of $\Omega^{*}_{U}$. Namely, these are projective toric varieties $L(i,j):=\P(\eta_{i}\oplus\underline{\C}^{j})\to\C P^{i}$ of complex dimension $i+j$.

\begin{cor}\label{cor:nspl}
$BR_{i,j}$ is not totally normally split for $j>2$. $L(i,j)$ is not totally normally split for $j>1$.
\end{cor}
\begin{Proof}
\end{Proof}
\subsection{Construction of representatives and necessary manifolds}\label{ss:const}

Remind that the blow-up of a non-singular projective toric variety $X$ at an invariant toric subvariety is again a projective toric variety. I present a particular case of equivariant blow-up which respects the property of being totally normally split. I use equivariant blow-ups of codimension $2$ subvarieties of totally normally split toric varieties to find quasitoric totally normally split manifolds with ``better'' Milnor numbers. 
\begin{pr}\label{pr:constr}
Let $X$ be a nonsingular projective complex variety of complex dimension $n$. Let $Z\subset X$ be a closed smooth subvariety in $X$ of (complex) codimension $2$. Suppose that $X$ is stably normally split. Then $Bl_{Z} X$ is also stably normally split.
\end{pr}
\begin{proof}
By the definition (see \cite[\S 6.2.1]{sh-13}), $Bl_{Z} X$ is a hypersurface in $ X\times\C P^{1}$. Denote the corresponding inclusion morphism by $\iota$. Let $\nu\to Bl_{Z} X$ be the normal (linear) vector bundle of this embedding. $\C P^{1}$ is stably splitting, so is $X\times \C P^{1}$. Hence, there exists a totally split vector bundle $\xi=\bigoplus_{i=1}^{k}\xi_{i}\to X\times \C P^{1}$, $\rk_{\C} \xi_{i}=1$, s.t.
\[
T(X\times \C P^{1})\oplus \xi\simeq \underline{\C}^{n+k+1}.
\]
Restrict this identity to $Bl_{Z} X$:
\[
T (Bl_{Z} X)\oplus \bigl(\nu\oplus\iota^{*}\xi\bigr)\simeq \underline{\C}^{n+k+1}.
\]
Q.E.D.
\end{proof}
\begin{rem}
Notice that the previous Proposition holds for an embedding $Z\subset X$ with non-split normal vector bundle.
\end{rem}
Remind that a smooth manifold $X$ of real dimension $2n$ is called \emph{quasitoric}, if it admits a smooth, locally standard action of $n$-dimensional torus, with orbit space an $n$-dimensional simple convex polytope. A quasitoric manifold is endowed with a natural stably complex structure (see \cite[\S 7.3]{bu-pa-15}). Hence, one can consider the unitary cobordism classes of quasitoric manifolds. We need the following fact about quasitoric manifolds (see \cite{bu-pa-ra-07}, \cite{bu-pa-15}).

\begin{pr}\normalfont{(\cite[\S 9.1]{bu-pa-15},\cite[Lemma 3.5]{ra-86}).}\label{pr:conn}
Let $M_{1},M_{2}$ be quasitoric $2n$-dimensional manifolds. Then there exist quasitoric manifolds $M, M'$ representing the unitary cobordism classes $-[M_{1}],[M_{1}]+[M_{2}]$, resp. Moreover, if $M_{1}$ ($M_{1},M_{2}$, resp.) is totally normally split, then $M$ ($M'$, resp.) is also totally normally split.
\end{pr}

For $1\leq i\leq j$ consider
\[
X_{i,j}=BF(\overline{\beta}_{i},\overline{\beta}^{*}_{i},\dots \overline{\beta}^{*}_{2},\overline{\beta}^{*}_{1},\underbrace{\underline{\C},\dots,\underline{\C}}_{j-i})\to BF_{i}.
\]
Notice that $X_{0,j}=BF_{j}$. For $i>0$ let $Z_{i,j}$ be the codimension $2$ subvariety of $X_{i,j}$ given by the conditions on the tautological line bundles: $\beta_{i}=\beta_{i-1}, \zeta_{j+1}=\zeta_{j}$. Denote $M_{i,j}=Bl_{Zi,j} X_{i,j}$.

\begin{pr}\label{pr:desc}
For $i>0$ one has:
\[
Z_{i,j}=BF(\underline{\C},\overline{\beta}^{*}_{i-1},\dots \overline{\beta}^{*}_{2},\overline{\beta}^{*}_{1},\underbrace{\underline{\C},\dots,\underline{\C}}_{j-i})\to BF_{i-1}.
\]
The normal bundle of the above inclusion $Z_{i,j}\subset X_{i,j}$ is equal to $\overline{\zeta_{j}}\overline{\beta_{i-1}}\oplus \overline{\beta_{i-1}}\to Z_{i,j}$.
\end{pr}
\begin{proof}
\end{proof}
\begin{pr}
$X_{i,j}$ and $M_{i,j}$ are toric totally tangentially and normally split manifolds.
\end{pr}
\begin{proof}
Due to the Corollary \ref{cor:bfsplit}, $BF_{i}$ is totally normally split. The formulas \eqref{eq:1.5},\eqref{eq:2} imply that $\overline{\beta}_{k},\overline{\beta^{*}}_{k},k=1,\dots,i$, have totally split stably inverse vector bundles. Then by Proposition \ref{pr:split} $X_{i,j}$ is totally normally splitting. The Proposition \ref{pr:constr} then implies that $M_{i,j}$ is also totally normally split. Remind that any smooth projective toric variety is always totally tangentially split (see \cite[Theorem 6.6]{da-ja-91} or \cite[Theorem 7.3.15]{bu-pa-15})
\end{proof}

The proof of the next statement is postponed until Section \ref{sec:compmi}.

\begin{pr}\label{pr:comp}
Let $1\leq i\leq j;\ n:=i+j$. Then one has:
\[
s_{n}(M_{i,j})=
\begin{cases}
(-1)^{n+1}\binom{n}{j}-\sum_{k=j}^{n-1}\binom{k}{j},\ \mbox{for } i\geq 2;\\
(-1)^{n+1}(n-1)-2,\ \mbox{for } i=1.
\end{cases}
\]
\end{pr}

For $n$ even, let $a_{0,n}:=s_{n}(-[M_{1,n}]),a_{1,n-1}=s_{n}([M_{1,n}])$. For $n$ odd, let $a_{0,n}:=s_{n}([M_{1,n}]+2[BF_{n}]),a_{1,n-1}=s_{n}([M_{1,n}]+[BF_{n}])$. It follows from \eqref{eq:1.1} and Proposition \ref{pr:comp} that $a_{0,n}=n+1$. For $0<i\leq j$ let $a_{i,j}:=s_{i+j}([M_{i,j}])$. That is, for $0<i\leq j$ one has:
\[
a_{i,j}=(-1)^{n+1}\binom{n}{j}-\sum_{k=j}^{n-1}\binom{k}{j},\ n=i+j.
\]
For any $0\leq i\leq j$, let $N_{i,j}$ be a quasitoric manifold representing the unitary cobordism classes above, so that $s_{n}([N_{i,j}])=a_{i,j}$, $i+j=n$. (See Proposition \ref{pr:conn}.)

\section{Proof of the Theorem}
In this Section the proof of Theorem \ref{thm:main} is given. We use standard facts from unitary cobordism theory (for example, see \cite{so-us-16}). An auxiliary statement from Number Theory is needed (see \cite{fi-47} for the proof).

\begin{thm}[Lucas]\label{lucas}
Let $p$ be prime, and let
$$
\begin{gathered}
n=n_{0}+n_{1}p+\dots+n_{r-1}p^{r-1}+n_{r}p^{r},\\
m=m_{0}+m_{1}p+\dots+m_{r-1}p^{r-1}+m_{r}p^{r}
\end{gathered}
$$
be the base $p$ expansions of the positive integers $n$ and $m$. Then one has
$$
\binom{n}{m}\equiv\binom{n_{0}}{m_{0}}\binom{n_{1}}{m_{1}}\dots\binom{n_{r}}{m_{r}}\pmod{p}.
$$
\end{thm}

The proof of the next technical Proposition is postponed until Section \ref{sec:nt}.
\begin{pr}\label{pr:pmain}
Let $s\geq 2$. Then for any prime $p$ one has:
\[
\sum_{k=p^{s}-p^{s-1}-1}^{p^{s}-1}\binom{k}{p^{s}-p^{s-1}-1}\equiv p\pmod {p^{2}}.
\]
\end{pr}

\begin{pr}
Suppose that $n=p^{s}-1$ for $s\geq 1$ and prime $p$. Then there exists a totally tangentially and normally split quasitoric manifold which represents a multiplicative generator in the unitary cobordism ring $\Omega^{*}_{U}$ in degree $2n$.
\end{pr}
\begin{proof}
If $s=1$, then the desired manifold is $N_{0,n}$ with Milnor number equal to $n+1=p$. Otherwise, for $p=2$ we take $BF_{n}$ with Milnor number equal to $2$ (see \eqref{eq:1.1}). Now suppose $s\geq 2,p>2$. Then $p^{s-1}<p^{s}-p^{s-1}-1$. According to Proposition \ref{pr:pmain}, $gcd(a_{0,n},a_{p^{s-1},p^{s}-p^{s-1}-1})=p$. Hence, the required quasitoric manifold with Milnor number $p$ can be constructed from $N_{0,n},N_{p^{s-1},p^{s}-p^{s-1}-1}$ (see Proposition \ref{pr:conn}).
\end{proof}

Further we use denotation $n=n_{0}+n_{1}p+\dots+n_{d}p^{d}=\overline{n_{d},\dots, n_{0}}_{p}$ in base $p$. Also the number of digits of the numbers written in base $p$ is $s$ below, so it is omitted.

\begin{pr}
Let $n$ be s.t. $n+1$ is not a power of a prime. Then there exists a quasitoric totally tangentially and normally split manifold which represents a multiplicative generator in the unitary cobordism ring $\Omega^{*}_{U}$ in degree $2n$.
\end{pr}
\begin{proof}
It is enough to show that for any prime divisor $q$ of $n+1$ there exists a quasitoric totally tangentially and normally split manifold with Milnor number not divisible by $q$. Quasitoric manifolds $N_{i,j}$ from the previous Section are used. 
Consider the following cases.

$1)$ $n=\overline{1, q-1\dots,q-1}_{q}$. Then $n+1=2q^{s-1}$ is even, and $q>2$ due to the assumption on $n$. Let
\[
j=\overline{1,q-1,\dots,q-1,0}_{q}.
\]
Notice that $n-j<j$. Using Lucas theorem one obtains $a_{n-j,j}\equiv 1-(q-1)\equiv2\not\equiv 0\pmod q$.

$2)$ $n=\overline{x_{s-1}, q-1\dots,q-1}_{q}, x_{s-1}>1$. Then $1<x_{s-1}<q-1$. Let
\[
j=x_{s-1}q^{s-1}-1=\overline{x_{s-1}-1,q-1,\dots,q-1}_{q}.
\]
Notice that $n-j=q^{s-1},$ so $n-j<j$, because $1<x_{s_{1}}$. Using Lucas theorem one obtains $a_{n-j,j}\equiv \pm x_{s-1}-1\not\equiv 0\pmod q$.

$3)$
\[
n=\overline{x_{s-1},\dots,x_{a},\dots,x_{b},q-1,\dots,q-1}_{q},
\]
where $0<x_{a};\ x_{b}<q-1;\ b<a;\ 0<x_{s-1}$ (so $q^{b}$ is the highest power of $q$ dividing $n+1$). Let
\[
j=\overline{x_{s-1},\dots,x_{a}-1,q-1\dots,q-1}_{q},
\]
where $x_{a}-1$ is in the $a$-th digit. Then
\[
n-j=\overline{0,\dots,x_{a-1},\dots,x_{b}+1,0,\dots,0}_{q}<j.
\]
By Lucas Theorem
\[
\binom{k}{j}\equiv 0\pmod q\ \mbox{for}\  j<k\leq n,
\]
and $\binom{j}{j}=1$. Hence, $a_{n-j,j}\equiv-1\pmod q$.
\end{proof}

\section{Computation of Milnor numbers}\label{sec:compmi}
In the following Section I compute the Milnor number of the previously defined toric varieties $X_{i,j}$ and their equivariant blow-ups $M_{i,j}$. For this purpose, some results are used (cf. \cite{so-us-16} for more details).

\begin{pr}[{Hitchin \cite[\S 4.5]{hi-74}}]\label{prop:hit}
Let $X$ and $Z$, with $Z\subset X$, be smooth compact complex manifolds of dimensions $n$ and $k$, resp. Consider a blow-up $\pi: Bl_{Z}X\rightarrow X$ along $Z$. Then the difference of classes of manifolds $Bl_{Z}X$ and $X$ in the unitary cobordism ring is:
\begin{equation}\label{eq:cobordism_blowup_identity}
[Bl_{Z}X] - [X] = -[\mathbb{P}(\nu(Z\subset X)\oplus\overline{\mathbb{C}})],
\end{equation}
where $\nu(Z\subset X)$ is a normal bundle to $Z$, and the projectivisation $\mathbb{P}(\nu(Z\subset X)\oplus\overline{\mathbb{C}})$ is equipped with the non-standard stably complex structure
\begin{equation}\label{eq:stably-complex-nonstandard}
T\mathbb{P}(\nu(Z\subset X)\oplus\overline{\mathbb{C}})\oplus\underline{\C}\simeq (p^*\nu(Z\subset X)\otimes \gamma)\oplus\gamma^*\oplus p^*TB,
\end{equation}
where $\gamma\to \mathbb{P}(\nu(Z\subset X)\oplus\overline{\mathbb{C}})$ is the corresponding tautological line bundle.
\end{pr}

Cohomology ring of $BF_{i}$ is easily computed using Leray-Hirsch Theorem:
\begin{pr}\normalfont{(See \cite[Theorem 7.8.2]{bu-pa-15}).}
One has an isomorphism of graded $\Z$-rings:
\[
H^{*}(BF_{i};\Z)\simeq\Z[t_{1},\dots,t_{i}]/(t_{a}^{2}-t_{a}t_{a-1}|\ a=1,\dots,i),
\]
where $t_{0}:=0$.
\end{pr}
The fundamental class of $BF_{i}$ is Poncar\'e dual to the element $t_{i}^{i}=t_{i}\cdots t_{1}\in H^{*}(BF_{i};\Z)$. One has the identity
\[
(1+t_{a})(1-t_{a}+t_{a-1})=1+t_{a-1}
\]
for $a=1,\dots,i$ in $H^{*}(BF_{i};\Z)$. Hence,
\begin{equation}\label{eq:id1}
(1+t_{i})\prod_{a=1}^{i}(1-t_{a}+t_{a-1})=1.
\end{equation}

Now we compute the Milnor numbers of the varieties $X_{i,j}$ introduced in Section \ref{ss:const}.
\begin{pr}\label{pr:snbott}
If $i+j$ is even, then $s_{i+j} (X_{i,j})=0.$ Otherwise,
\[
s_{i+j}(X_{i,j})=2\binom{i+j}{i}.
\]
\end{pr}
\begin{proof}
Use Proposition \ref{pr:snbfb}. Case of even $i+j$ is trivial. Suppose that $i+j$ is odd. Then
\[
s_{i+j}(X_{i,j})=2\la(1+t_{i})^{i+j}\bigl((1+t_{i})\prod_{a=1}^{i}(1-t_{a}+t_{a-1})\bigr)^{-1}, BF_{i}\ra\overset{\eqref{eq:id1}}{=}2\la(1+t_{i})^{i+j}, BF_{i}\ra=2\binom{i+j}{i}.
\]
Q.E.D.
\end{proof}

\begin{cor}
\begin{equation}\label{eq:1.1}
s_{n}(BF_{n})=1+(-1)^{n+1}.
\end{equation}
\end{cor}
\begin{proof}
Apply the previous Proposition for $i=0$.
\end{proof}

Let $Z_{i+j-2}=Z_{i,j}$ and denote the corresponding ``bases'' of complex dimension $k$ of the Bott tower $Z_{i,j}$ by $Z_{k}$ (see Subsection \ref{ss:const}). Notice that for $i>0$, $Z_{i-1}=BF_{i-1}$. For $i>0$ let $Y_{i,j}:=\P(\overline{\zeta_{j}}\overline{\beta_{i-1}}\oplus \overline{\beta_{i-1}}\oplus\overline{\C})\to Z_{i,j}$. 

\begin{pr}\label{pr:sndiff}
\[
s_{i+j}(Y_{i,j})=
\begin{cases}
\sum_{k=j}^{i+j}\binom{k}{j},\ \mbox{for } i\geq 2;\\
j+1+(-1)^{j+1},\ \mbox{for } i=1.
\end{cases}
\]
\end{pr}
\begin{proof}
Let $y_{k}:=c_{1}(\overline{\zeta_{k}})$ for $\zeta_{k}\to Z_{i,j}$, and let $y=c_{1}(\overline{\zeta})$ for the tautological line bundle $\zeta$ over $Y_{i,j}$ (remind that $Y_{i,j}$ is a projective bundle over $Z_{i,j}$). For the tangent bundle,
\[
T Y_{i,j}\oplus\underline{\C}\cong \overline{\zeta}\overline{\zeta_{j}}\overline{\beta_{i-1}}\oplus \overline{\zeta}\overline{\beta_{i-1}}\oplus\zeta.
\]
Then
\[
s_{i+j}(Y_{i,j})=\la(y+x_{i-1}+y_{j})^{i+j}+(y+x_{i-1})^{i+j}+(-y)^{i+j}, Y_{i,j}\ra.
\]
First, suppose that $i>1$. Below I conduct computations using Segre class.

\begin{multline}
\la (y+x_{i-1}+y_{j})^{i+j}, Y_{i,j}\ra=\la (1+x_{i-1}+y_{j})^{i+j-1}(1+x_{i-1})^{-1}, Z_{i+j-2}\ra=\\
=\la\sum_{k=0}^{i+j-1}\binom{i+j-1}{k}(1+x_{i-1})^{k-1}y_{j}^{i+j-1-k},Z_{i+j-2}\ra=
\la\sum_{k=1}^{i+j-1}\binom{i+j-1}{k}x_{i-1}^{k-1}y_{j}^{i+j-1-k},Z_{i+j-2}\ra=\\
\la\bigl(\sum_{k=1}^{i+j-1}\binom{i+j-1}{k}x_{i-1}^{k-1}\bigr)(1-x_{i-1}+x_{i-2})^{-1}\cdots(1-x_{1})^{-1},BF_{i-1}\ra \overset{\eqref{eq:id1}}{=}\\
=\la\bigl(\sum_{k=1}^{i+j-1}\binom{i+j-1}{k}x_{i-1}^{k-1}\bigr)(1+x_{i-1}),BF_{i-1}\ra=\binom{i+j-1}{i}+\binom{i+j-1}{i-1}=\binom{i+j}{i}.
\end{multline}

\begin{multline}
\la(y+x_{i-1})^{i+j}, Y_{i,j}\ra=\la(1+x_{i-1})^{i+j-2}(1+y_{j}(1+x_{i-1})^{-1})^{-1}, Z_{i+j-2}\ra=\\
=\la\sum_{k=0}^{i+j-2}(-1)^{k}(1+x_{i-1})^{i+j-2-k}y_{j}^{k}, Z_{i+j-2}\ra=
\la\sum_{k=0}^{i+j-2}(-1)^{k}x_{i-1}^{i+j-2-k}y_{j}^{k}, Z_{i+j-2}\ra=\\
=\la(1+x_{i-1})^{-1}(1+y_{j-1})^{-1}, Z_{i+j-3}\ra=
\la(1+x_{i-1})^{-1}(1-x_{i-1}+x_{i-2})^{-1}(1-y_{j-2})^{-1}, Z_{i+j-4}\ra=\dots=\\
=\la(1+x_{i-1})^{-1}(1-x_{i-1}+x_{i-2})^{-1}\cdots(1-x_{1})^{-1}, BF_{i-1}\ra\overset{\eqref{eq:id1}}{=}
\la(1+x_{i-1})^{-1}(1+x_{i-1}), BF_{i-1}\ra=\\
=(-1)^{j-1}+(-1)^{j}=0.
\end{multline}


\begin{multline}
(-1)^{i+j}\la y^{i+j}, Y_{i,j}\ra=(-1)^{i+j}\la (1+x_{i-1})^{-1}(1+x_{i-1}+y_{j})^{-1}, Z_{i+j-2}\ra=\\
=\la \sum_{a=0}^{i+j-2}x_{i-1}^{a}(x_{i-1}+y_{j})^{i+j-2-a}, Z_{i+j-2}\ra=
\la \sum_{a=0}^{i+j-2}\sum_{b=0}^{i+j-2-a}\binom{i+j-2-a}{b}x_{i-1}^{a+b}y_{j}^{i+j-2-a-b}, Z_{i+j-2}\ra=\\
=\dots=\la \bigl(\sum_{a=0}^{i+j-2}x_{i-1}^{a}(1+x_{i-1})^{i+j-2-a}\bigr)(1-x_{i-1}+x_{i-2})^{-1}\cdots(1-x_{1})^{-1}, BF_{i-1}\ra\overset{\eqref{eq:id1}}{=}\\
=\la \bigl(\sum_{a=0}^{i+j-2}x_{i-1}^{a}(1+x_{i-1})^{i+j-2-a}\bigr)(1+x_{i-1}), BF_{i-1}\ra=
\sum_{a=0}^{i-1}\binom{i+j-1-a}{i-1-a}=\sum_{a=0}^{i-1}\binom{j+a}{j}.
\end{multline}
Now the claim follows for $i>1$.

Suppose that $i=1$. By Proposition \ref{pr:desc}, $Z_{1,j}=BF_{j-1}$. So:
\begin{multline}
s_{1+j}(Y_{1,j})=\la(y+y_{j})^{1+j}+(1+(-1)^{j+1})y^{1+j},Y_{1,j}\ra=\\
=\la(1+y_{j})^{j}+(1+(-1)^{j+1})(1+y_{j})^{-1},BF_{j-1}\ra=
j+1+(-1)^{j+1}.
\end{multline}

Q.E.D.
\end{proof}

\begin{proof}[Proof of Proposition~\ref{pr:comp}]
Follows from the previous Propositions \ref{prop:hit}, \ref{pr:snbott}, \ref{pr:sndiff} and additivity of the Milnor number.
\end{proof}

\section{Number-theoretical computations}\label{sec:nt}
In this Section the proof of Proposition \ref{pr:pmain} is given. 

\begin{lm}
Let $0\leq r<p$ be an integer. Then:
\begin{equation}\label{eq:1.2}
\prod_{k=1}^{p-1}(pr+k)\equiv (p-1)! \pmod {p^{2}}.
\end{equation}
\end{lm}
\begin{proof}
Follows from the computation:
\begin{equation}\label{eq:1.7}
(pr+k)(pr+p-k)\equiv k(p-k)+p(kr+(p-k)r)\equiv k(p-k)\pmod {p^{2}},
\end{equation}
where $0\leq r<p$.
\end{proof}


%

\begin{lm}\label{lm:tech}
For any $0<a<p$ one has:
\[
\frac{\prod_{k=1}^{a}(p(p-1)+k)}{a!}\equiv 1-p\sum_{k=1}^{a}\frac{1}{k} \pmod {p^{2}}.
\]
\end{lm}
\begin{proof}
First, observe that
\begin{equation}\label{eq:1.6}
\prod_{k=1}^{a}(k-p)\equiv a!\biggl(1-p\sum_{k=1}^{a}\frac{1}{k}\biggr)\pmod {p^{2}}.
\end{equation}
So
\[
\frac{\prod_{k=1}^{a}(p(p-1)+k)}{a!}\equiv \frac{\prod_{k=1}^{a}(k-p)}{a!}\equiv
1-p\sum_{k=1}^{a}\frac{1}{k} \pmod {p^{2}}.
\]
\end{proof}

\begin{lm}\label{lm:tech4}
For any $0<a<p$ one has:
\begin{equation}\label{eq:1.9}
\sum_{a=1}^{p-1}\frac{\prod_{k=1}^{a-1}(p(p-1)+k)}{a!}\equiv 0\pmod {p^{2}}.
\end{equation}
\end{lm}
\begin{proof}
To prove the claim of the Lemma it is enough to show that the summands in \eqref{eq:1.9} for $a$ and $p-a$ are additively inverse to each other. First observe that
\begin{equation}\label{eq:1.10}
p\biggl(\frac{1}{a}+\frac{1}{p-a}\biggr)\equiv 0\pmod {p^{2}}.
\end{equation}
By Lemma \ref{lm:tech} one has:
\[
\frac{\prod_{k=1}^{a-1}(p(p-1)+k)}{a!}\equiv\frac{1}{a+p(p-1)}\biggl(1-p\sum_{k=1}^{a}\frac{1}{k}\biggr)\equiv
\frac{1}{a-p}\biggl(1-p\sum_{k=1}^{a}\frac{1}{k}\biggr)\pmod {p^{2}},
\]
\begin{multline}
\frac{\prod_{k=1}^{p-a-1}(p(p-1)+k)}{(p-a)!}\equiv\frac{1}{p-a}\biggl(1-p\sum_{k=1}^{p-a-1}\frac{1}{k}\biggr)\overset{\eqref{eq:1.10}}{\equiv}
\frac{1}{p-a}\biggl(1+p\sum_{k=p-a}^{p-1}\frac{1}{k}\biggr)\overset{\eqref{eq:1.10}}{\equiv}\\
\equiv\frac{1}{p-a}\biggl(1-p\sum_{k=1}^{a}\frac{1}{k}\biggr) \pmod {p^{2}}.
\end{multline}
Q.E.D.
\end{proof}

Another facts from Number Theory are required. Denote by $k!_{p}$ the product of all consequent integers from $1$ to $k$ not divisible by $p$.
\begin{thm}\normalfont{(See \cite[Theorem 1]{gra-97}).}\label{thm:nt}
Suppose that prime power $p^q$ and positive numbers $m=n+r$ are given. Write $n=n_{0}+n_{1}p+\dots+n_{d}p^{d}$ in base $p$, and let $N_{i}$ be the least positive residue of $[n/p^{j}]\ \pmod  p^{q}$ for each $j\geq 0$ (so that $N_{j}=n_{j}+n_{j+1}p+\dots+n_{j+q-1}p^{q-1}$): also make the corresponding definitions for $m_{j},M_{j},r_{j},R_{j}$. Let $e_{j}$ be the number of indices $i\geq j$ for which $n_{i}<m_{i}$ (that is, the number of ``carries'', when adding $m$ and $r$ in base $p$, on or beyond the $j$-th digit). Then
\[
\frac{1}{p^{e_{0}}}\binom{n}{m}\equiv (\pm 1)^{e_{q-1}}\biggl(\frac{(N_{0})!_{p}}{(M_{0})!_{p}(R_{0})!_{p}}\biggr)
\biggl(\frac{(N_{1})!_{p}}{(M_{1})!_{p}(R_{1})!_{p}}\biggr)\cdots\biggl(\frac{(N_{d})!_{p}}{(M_{d})!_{p}(R_{d})!_{p}}\biggr)
\pmod p^{q},
\]
where $(\pm 1)$ is $(-1)$ except $p=2$ and $q\geq 3$.
\end{thm}

\begin{lm}\normalfont{(\cite[\S 2, Lemma 1]{gra-97}).}\label{lm:ressq}
\begin{equation}\label{eq:1.3}
(p^{2})!_{p}\equiv -1\pmod {p^{2}}.
\end{equation}
\end{lm}

\begin{lm}\label{lm:3.2}
For the modulus $p^{2}$ residue of the binomial number one has:
\[
\binom{p^{s}-1}{p^{s}-p^{s-1}-1}\equiv p-1 \pmod {p^{2}}.
\]
\end{lm}
\begin{proof}
Clearly,
\[
\overline{p-1, p-1,\dots, p-1}_{p}=\overline{p-2, p-1,\dots, p-1}_{p}+\overline{1, 0\dots, 0}_{p}.
\]
By Theorem \ref{thm:nt} then one has:
\begin{multline}
\binom{p^{s}-1}{p^{s}-p^{s-1}-1}=\binom{\overline{p-1, p-1,\dots, p-1}_{p}}{\overline{p-2, p-1,\dots, p-1}_{p}}\equiv
(p-1)\frac{(p-1+p(p-1))!_{p}}{(p-1+p(p-2))!_{p}p!_{p}}\equiv \\
\equiv(p-1)\frac{\prod_{k=0}^{p-1} (p(p-1)+k)}{p!_{p}}\equiv p-1
\pmod {p^{2}}.
\end{multline}

\end{proof}

\begin{lm}\label{lm:3.1}
\[
\sum_{k=p^{s}-p^{s-1}}^{p^{s}-2}\binom{k}{p^{s}-p^{s-1}-1}\equiv 0\pmod {p^{2}}.
\]
\end{lm}
\begin{proof}
\begin{equation}\label{eq:3.1}
\sum_{k=p^{s}-p^{s-1}}^{p^{s}-2}\binom{k}{p^{s}-p^{s-1}-1}=\sum_{k=0}^{s-2}\sum_{\substack{x_{k},\dots,x_{s-2}=0\\ x_{k}<p-1}}^{p-1}
\binom{\overline{p-1, x_{s-2},x_{s-3},\dots,x_{k},p-1,\dots, p-1}_{p}}{\overline{p-2,p-1, p-1,\dots, p-1,p-1,\dots,p-1}_{p}}.
\end{equation}
By Kummer Theorem, if at least two of $x_{k},\dots,x_{s-2}$ are different from $p-1$ then
\[
\binom{\overline{p-1, x_{s-2},x_{s-3},\dots,x_{k},p-1,\dots, p-1}_{p}}{\overline{p-2,p-1, p-1,\dots, p-1,p-1,\dots,p-1}_{p}}\equiv 0 \pmod {p^{2}}.
\]
Hence,
\begin{multline}\label{eq:3.2}
\sum_{k=0}^{s-2}\sum_{\substack{x_{k},\dots,x_{s-2}=0\\ x_{k}<p-1}}^{p-1}
\binom{\overline{p-1, x_{s-2},x_{s-3},\dots,x_{k},p-1,\dots, p-1}_{p}}{\overline{p-2,p-1, p-1,\dots, p-1,p-1,\dots,p-1}_{p}}\equiv\\
\equiv\sum_{k=0}^{s-2}\sum_{x_{k}=0}^{p-2}
\binom{\overline{p-1, p-1,\dots,x_{k},p-1,\dots, p-1}_{p}}{\overline{p-2,p-1,\dots, p-1,p-1,\dots,p-1}_{p}}\pmod {p^{2}},
\end{multline}
where $x_{k}$ is in the $k$-th digit. Let $k=s-2$, then Theorem \ref{thm:nt} and Lemmas \ref{lm:tech4}, \ref{lm:ressq} one has:
\begin{multline}\label{eq:3.3}
\frac{1}{p}\cdot\sum_{x_{s-2}=0}^{p-2}
\binom{\overline{p-1, x_{s-2},p-1,\dots, p-1}_{p}}{\overline{p-2,p-1,p-1,\dots,p-1}_{p}}\equiv
\pm\sum_{x_{s-2}=0}^{p-2}(p-1)\frac{(x_{s-2}+p(p-1))!_{p}}{(p-1+p(p-2))!_{p}(x_{s-2}+1)!}\equiv\\
\equiv\pm(p-1)\sum_{a=1}^{p-1}\frac{\prod_{r=1}^{a-1}(p(p-1)+r)}{a!}\equiv 0\pmod {p^{2}}.
\end{multline}

Now let $0\leq k<s-2$. Then Theorem \ref{thm:nt} and Lemmas \ref{lm:tech4}, \ref{lm:ressq} imply that:
\begin{multline}
\frac{1}{p}\cdot\sum_{x_{k}=0}^{p-2}
\binom{\overline{p-1, p-1,\dots,p-1,x_{k},p-1,\dots, p-1}_{p}}{\overline{p-2,p-1,\dots,p-1, p-1,p-1,\dots,p-1}_{p}}\equiv\\
\equiv\pm\sum_{x_{k}=0}^{p-2}(p-1)\frac{(x_{k}+p(p-1))!_{p}}{(p-1+p(p-1))!_{p}(x_{k}+1+p(p-1))!}\equiv\pm(p-1)\sum_{a=1}^{p-1}\frac{1}{a-p}\equiv\\
\equiv \pm(p-1)\sum_{a=1}^{(p-1)/2}\frac{p}{a(a-p)}\pmod {p^{2}}.
\end{multline}
Hence,
\begin{equation}\label{eq:3.4}
\sum_{x_{k}=0}^{p-2}
\binom{\overline{p-1, p-1,\dots,p-1,x_{k},p-1,\dots, p-1}_{p}}{\overline{p-2,p-1,\dots,p-1, p-1,p-1,\dots,p-1}_{p}}\equiv 0\pmod {p^{2}}.
\end{equation}
Now the claim follows from the identities \eqref{eq:3.1},\eqref{eq:3.2},\eqref{eq:3.3},\eqref{eq:3.4}.
\end{proof}
\begin{proof}[Proof of Proposition \ref{pr:pmain}]
Follows from Lemmas \ref{lm:3.2}, \ref{lm:3.1}.
%
%

\end{proof}

\section{Concluding remarks}\label{sec:conc}
A possibility of a more elegant proof of Theorem \ref{thm:main} still remains. One can try to find a simpler family of quasigenerators of $\Omega^{*}_{U}$ which are quasitoric and stably normally splitting manifolds. I give a possible example motivated by N.~Ray's research \cite{ra-86}. However it requires a deeper study.

\begin{definition}\label{def:sij}
For $0\leq i\leq j$ let $S_{i,j}$ be the hypersurface of $BF_{i}\times BF_{j}$ given by the equation
\begin{equation}\label{eq:rij}
\sum_{k=0}^{i}z_{i,k}w_{j,k+j-i}=0.
\end{equation}
In particular, $S_{0,j}=BF_{j-1}$. (Notice, that $S_{0,0}=\varnothing$.)
\end{definition}

\begin{rem}\label{rem:sing}
The ordering of $w$-variables in \eqref{eq:rij} is crucial. For example, consider the subvariety $S'_{1,2}\subset BF_{1}\times BF_{2}$ given by the equation
\[
z_{1,0}w_{2,0}+z_{1,1}w_{2,1}=0.
\]
Remind that $BF_{1}\times BF_{2}\subset \C P^{1}\times\C P^{1}\times\C P^{2}$ is given by the only equation
\[
w_{2,0}w_{1,1}-w_{2,1}w_{1,0}=0.
\]
Now one can easily check that $S'_{1,2}$ is singular along the subvariety $\lb w_{2,0}=w_{2,1}=0\rb$ isomorphic to $\C P^{1}$.
\end{rem}
Observe that $S_{0,j}=BF_{j-1}$. Using similar arguments from \cite{ra-86} one can show that $S_{i,j}$ are totally normally split. $S_{i,j}$ is dual to the linear vector bundle $\overline{\beta_{i}}\overline{\beta'_{i}}\to BF_{i}\times BF_{j}$. Thus, for $0<i\leq j,\ s_{i+j-1}(S_{i,j})=-\binom{i+j}{i}$. Hence, together with $N_{0,n}$ these manifolds ($i+j=n+1$) form a family of multiplicative generators of $\Omega^{*}_{U}$ in degree $2n$. One can see that $S_{i,j}$ is a non-singular projective algebraic variety obtained by sequential blow-ups of strict transforms of some subvarieties of $BR_{i,j}$. However, these subvarieties seem to be not invariant under the natural torus action on $BR_{i,j}$. So, here is the question.

{\medskip\noindent\bf Problem.}
Is $S_{i,j}$ a toric variety for all $0\leq i\leq j$?

\end{document}